\documentclass[12pt]{amsart}
\usepackage{amsfonts}
\usepackage{amsmath}
\usepackage{amssymb}
 \usepackage{dsfont}

\newtheorem{theorem}{Theorem}[section]
\newtheorem{lemma}[theorem]{Lemma}

\newtheorem{conjecture}[theorem]{Conjecture}

\theoremstyle{definition}
\newtheorem{definition}[theorem]{Definition}

\theoremstyle{remark}

\author{Wang Feng}

\begin{document}

\title{ A volume  stability theorem  on toric manifolds  with positive Ricci curvature }

\date{}

\maketitle
\begin{abstract}
In this short note, we will prove a volume stability theorem which says that if an n-dimensional toric manifold $M$ admits a $\mathbb{T}^n$ invariant K\"ahler metric $\omega$ with Ricci curvature no less than $1$ and its volume is close to the volume of $\mathbb{CP}^n$, $M$ is bi-holomorphic to $\mathbb{CP}^n$.
\end{abstract}
\section{\textbf{Introduction}}
To understand the geometry of manifolds under various curvature conditions is a fundamental question. In Riemannian geometry, we have Bishop-Gromov's volume comparison if the  Ricci curvature of the manifold is bounded from below. Using this theorem and some techniques in comparison geometry,  Colding proved the following result (\cite{5}):

 \begin{theorem}
 Given $\epsilon>0$, there exists $\delta = \delta(n,\epsilon) > 0$ such
that, if an n-dimensional manifold $M$ has $\text{Ric}_M\geq n-1$ and $\text{Vol}(M)>\text{Vol}(\mathbb{S}^n)-\delta,$ then $d_{GH}(M,\mathbb{S}^n)<\epsilon$.
 \end{theorem}
 Here $d_{GH}$ denotes the  Gromov-Hausdorff distance between  Riemannian manifolds.  By another theorem of Colding (see the appendix in \cite{4}), we know that $M$ is in fact  diffeomorphic to $\mathbb{S}^n$.

 It is a natural question that how to get a more useful version of  Bishop-Gromov's volume comparison theorem in K\"ahler geometry and how to state a theorem  analogous to the above one. Because we have more structures on the manifold, the volume comparison with space form is not sharp: see \cite{8} for an improvement of local volume comparison for K\"ahler manifolds with Ricci curvature bounded from below. More recently,  Berman and Berndtsson considered toric manifolds  with positive Ricci curvature in  \cite{2} and \cite{3}, and they proved

 \begin{theorem}\label{Thm-2}
 Suppose that $(M,\omega)$ is smooth n-dimensional toric variety with $\mathbb{T}^n$ invariant K\"ahler form $\omega$ such that Ric $\omega\geq\omega$, then we have
 \begin{align}
 \text{Vol}(M)\leq\text{Vol}(\mathbb{CP}^n).
 \end{align}
 \end{theorem}

In fact, their theorem holds if the manifold admits a $\mathbb{C}^*$ action with finite fixed points and the metric is $\mathbb{S}^1$ invariant (see \cite{3}). The theorem  of  Berman and Berndtsson partially answered  a conjecture in \cite{10}:

\begin{conjecture}\label{conj} Any n-dimensional toric Fano manifold X that admits a K\"ahler-Einstein metric has anticanonical degree
$(-K_X)^n \leq(n + 1)^n$, with equality only for $\mathbb{CP}^n$.
\end{conjecture}

  In this short note,  we will determine when the equality holds  in  Theorem \ref{Thm-2} and in particular we give a complete answer to Conjecture \ref{conj}. More precisely, we  can prove a rigidity and stability theorem as follows:

 \begin{theorem}\label{Thm-4} The equality  in  Theorem \ref{Thm-2} holds if and only if $(M,\omega)$ is isometric to $(\mathbb{CP}^n,\omega_{FS})$. Moreover, there exists a positive number $\epsilon$ which depends only on $n$ such that if $M$ is a toric manifold with a $\mathbb{T}^n$ invariant metric $\omega$ satisfying $\text{Ric }\omega\geq\omega$ and
\begin{align}
\text{Vol}(M)\geq(1-\epsilon)\text{Vol}(\mathbb{CP}^n),
\end{align} $M$ is bi-holomorphic to $\mathbb{CP}^n$.
\end{theorem}

In \cite{3},  Berman and Berndtsson  applied  a Moser-Trudinger typed  inequality established  in  \cite{1} to prove  Theorem \ref{Thm-2} . But so far we can't prove the rigidity using this analytic method.
 Inspired by the combinatoric proof by Bo'az Klartag for the K\"ahler-Einstein case in \cite{3}, we will apply the Grunbaum's inequality (\cite{7}) to prove our theorem. In order to use this inequality we should know the position of the barycenter of the moment polytope of $(M,\omega)$. We will use the Ricci curvature condition to achieve this. More detailed analysis gives us the rigidity and stability.

\textbf{Acknowledgement}: The author is grateful to his advisor Professor Zhu Xiaohua for many helpful conversations.
\section{\textbf{Preliminaries}}
At first, we give some basic materials of toric manifolds which are used in our proof. Here a toric manifold means a K\"ahler manifold $(M,\omega)$ containing $(\mathbb{C}^*)^n$ as a dense subset such that the standard action of $(\mathbb{C}^*)^n$ on itself extends to a holomorphic action on $M$. In general we suppose that the metric is $\mathbb{T}^n$ invariant and we can consider the moment map of $(M,\omega,\mathbb{T}^n)$.
\begin{definition}
 A polytope $P\subseteq\mathbb{R}^n$ is called a Delzant polytope if  each vertex is
contained in exactly n facets, and the normals of the n facets containing a given vertex form an integral basis of $\mathbb{Z}^n$.
\end{definition}
The image of the moment map above should be a Delzant polytope according to a theorem of Delzant (\cite{6}):
\begin{theorem}
Each Delzant polytope gives rise to a symplectic manifold $(M,\omega)$
with an action of $\mathbb{T}^n$ that preserves $\omega$, and all such symplectic manifolds arise this
way.
\end{theorem}
Using the embedding of $(\mathbb{C}^*)^n$ in $M$, we set:
\begin{align}
\iota:(\mathbb{C}^*)^n \rightarrow M,\\
\iota^*\omega=\sqrt{-1}\partial\bar{\partial}u.
\end{align}. In toric coordinates: $\exp(x_i)=|z_i|^2 (z_i \text{ are holomorphic coordinates in }(\mathbb{C}^*)^n) $, the invariance of $\omega$ means $u$ is function of $x_i$. Then the image of $\nabla u=(\frac{\partial u}{\partial x_i})_{1\leq i\leq n}$ will be a moment map of $(M,\omega)$.

Given a toric manifold $(M,\omega)$, two moment maps may differ by a constant vector. When we choose a basis of group $(\mathbb{C}^*)^n $, these two moment polytopes differ by a translation. A change of basis of group $(\mathbb{C}^*)^n$ corresponds to a change of the integral basis of $\mathbb{Z}^n$, so it transforms Delzant polytopes to Delzant polytopes. The polytope also changes if we choose another $\mathbb{T}^n$ invariant K\"ahler metric on $M$ with the same complex structure, i.e. we choose another symplectic form compatible with the fixed complex structure. This can be described in the following way:
we denote the moment polytope by
\begin{align}
 P=\{x| \langle l_i,x\rangle\geq \lambda_i, 1\leq i\leq N, x\in\mathbb{R}^n, l_i\in\mathbb{Z}^n, \lambda_i \in \mathbb{R}\}.
 \end{align}
Then only $\lambda_i$ $(1\leq i\leq n)$ change while $l_i$ $(1\leq i\leq n)$  remain the same since they are just related to the complex structure (see \cite{9}\cite{11}). Using the description above, changing the symplectic form corresponds to changing the potential function $u$ on $\mathbb{R}^n$.

 When the manifold is a Fano variety with $\omega\in 2\pi c_1(M)$, we can get a moment polytope $P$ such that $\lambda_i$ are all equal to $-1$. It's can be realized in the following way (see \cite{9}): choose a potential $u$ of $\omega$ such that
 \begin{align}
 |\ln\text{det}u_{ij}+u| \text{ is bounded in }\mathbb{R}^n,
 \end{align}
 then the image of $\nabla u$ will be such a polytope $P$ with $\lambda_i=-1 (1\leq i\leq n)$.

 Because the normal vectors of the facets passing any point form an integral basis, we can do a coordinate transformation to change these vectors to the standard basis $e_k=(0,0,...,1,0,...0)$ with $1$ placed at position $k$. We can write this transformation as follows: choosing a vertex $p\in P$ with $l_i (1\leq i\leq n)$ as normal vectors of the facets passing $p$, we can form an affine map:
 \begin{align}
 x\mapsto (\langle l_i,x\rangle)_{1\leq i\leq n},
 \end{align}
which transforms $p$ to $(-1,-1,...,-1)$ and the polytope to
\begin{align}
\tilde{P}=\{x| \langle \tilde{l}_i,x\rangle\geq -1, 1\leq i\leq N, x\in\mathbb{R}^n, \tilde{l}_i\in\mathbb{Z}^n, \tilde{l}_k=e_k, 1\leq k\leq n \}.
\end{align}There are only finite many such polytopes in a given dimension.

According to Mabuchi's theorem (\cite{9}), we know that for a K\"ahler-Einstein manifold, the origin is the barycenter of $P$. We will prove a similar property of the barycenter of the moment map of a toric manifold admitting $\omega$ with $\text{Ric }\omega\geq\omega$.

\section{\textbf{Proof of the theorem}}
At first we give a lemma which deals with the volume of some specific kind of polytopes. Let $Q$ be the simplex spanned by
\begin{align}
(n+1,0,0,...,0),(0,n+1,0,0,...,0),...,(0,0,...,0,n+1).
 \end{align}
Recall that the Grunbaum's inequality (\cite{7}) says that if $P$ is a convex body, and $K$ denotes the intersection of $P$ with an affine half-space defined by one side of a hyperplane $H$ passing through the barycenter of $P$, then
\begin{align}
\text{Vol}(P)\leq(\frac{n+1}{n})^n\text{Vol}(K).
\end{align}
Let $F$ denote the simplex spanned by $(n,0,0,...,0),(0,n,0,0,...,0),...,(0,0,...,0,n)$. We have the following lemma.
\begin{lemma}If the barycenter of a polytope $P$ in the first quadrant lies inside $F$,
\begin{align}
\text{Vol}(P)\leq\text{Vol}(Q).
\end{align}
Moreover if the equality holds, $P$ is coincident with $Q.$
\end{lemma}
\begin{proof}
 The first statement can be seen from Grunbaum's theorem above: the corresponding $K\subseteq F$. For the second statement, let $X=P\diagdown Q, Y=Q\diagdown P$ and choose a coordinate system $s_i$ with the barycenter as the origin and $(1,1,...,1)$ as the first axis. Then we have
 \begin{align}
  \int_Ps_1dV\leq\int_Qs_1 dV=0, \int_Xs_1dV\leq\int_Ys_1 dv.
  \end{align}
 But since
 \begin{align}
 s_1(x)\geq s_1(y) \text{ for } x\in X \text{ and }y\in Y,
 \end{align}
  both $X$ and $Y$ should be empty.
 \end{proof}

In order to apply this lemma to the moment polytope $P$ of $(M,\omega)$, we should know how to place $P$ and where the barycenter is. We are going to use the toric structure on $M$ and explore the Ricci curvature condition.

 Under the condition of theorem\ref {Thm-4}, we can write $\text{Ric }\omega=\omega+\beta$ where $\beta$ is a semi-positive 1-1 form. In $(\mathbb{C}^*)^n$, we can choose $u$ such that $\omega=\sqrt{-1}\partial\bar{\partial} u.$ In toric coordinates: $|z_i|^2=\exp(x_i)$, we set:
\begin{align}
v=-\ln\text{det} u_{ij}-u.
\end{align}
Using the formula of Ricci curvature and $\frac{\partial^2u}{\partial z_i\partial\bar{ z_j}}=\frac{\partial^2 u}{\partial x_i\partial x_j}\frac{1}{z_i}\frac{1}{\bar{z_j}}$, we see that
\begin{align}
\sqrt{-1}\partial\bar{\partial}v=\text{Ric }\omega-\omega=\beta.
\end{align} As $\beta$ is semi-positive, $v$ is a convex function.

From the following equalities:
\begin{align}
 \ln\text{det}(u+v)_{ij}+u+v=\ln\text{det}(u+v)_{ij}-\ln\text{det} u_{ij}=\ln\frac{(\text{Ric }\omega)^n}{\omega^n},
 \end{align}
 we know that $\ln\text{det}(u+v)_{ij}+u+v$ is bounded, so $\nabla (u+v)$ will be a moment map of $(M,\text{Ric }\omega)$. Denote the image of $\nabla (u+v)$ by $L$. As illustrated in section 2, we can suppose that $(-1,-1,...,-1)$ is a vertex of $L$ and the facets passing it are parallel to coordinate hyperplanes repectively:
 \begin{align}
 L=\{y| \langle l_i,y\rangle\geq -1, 1\leq i\leq N, y\in\mathbb{R}^n, l_i\in\mathbb{Z}^n, l_k=e_k, 1\leq k\leq n \}.
 \end{align}

 The gradient of $u$ will a moment of $(M, \omega)$. We denote the image of $\nabla u$ by $P$. Without changing $u+v$, we can add a linear function to $u$ and subtract the same one from $v$. This corresponds to a translation of $P$. As we have said above, $P$ can be obtained from $L$ by parallel movement of the facets. So we can translate $P$ so that $(-1,-1,...,-1)$ is a vertex of $P$ and the facets passing this vertex are parallel to coordinate hyperplanes like $L$:
 \begin{align}
 P=\{y| \langle l_i,y\rangle\geq \lambda_i, 1\leq i\leq N, y\in\mathbb{R}^n, l_i\in\mathbb{Z}^n, l_k=e_k, \lambda_k=-1, 1\leq k\leq n \}.
 \end{align}
Such a pair of polytopes $(P,L)$ is called an adapted pair of $(M,\omega)$.

  \begin{lemma}
 For an adapted pair $(P,L)$, the coordinates of the barycenter of $P$ are all nonpositive.
  \end{lemma}
  \begin{proof}
  \begin{align}
  \lim_{x_i\rightarrow-\infty}\frac{\partial v}{\partial x_i}=\lim_{x_i\rightarrow-\infty}\frac{\partial (u+v)}{\partial x_i}-\lim_{x_i\rightarrow-\infty}\frac{\partial u}{\partial x_i}=(-1)-(-1)=0
  \end{align}
  $\text{ for any }i \text{ and fixed } x_j ( 1\leq j\leq n, j\neq i)$.
  Because $v$ is convex function we know that all the partial derivatives of $v$ are nonnegative. Denote the coordinates of the barycenter by $a_i$, we have
 \begin{flalign}
 \text{det}u_{ij}=exp(-u-v), \frac{\partial v}{\partial x_i}\geq0,
\end{flalign}
\begin{flalign}
a_i=\int_Py_idV=\int_{R^n}\frac{\partial u}{\partial x_i}\text{det}u_{ij}dx
 \leq\int_{R^n}\frac{\partial (u+v)}{\partial x_i}exp(-u-v)dx=0.&
 \end{flalign}
 The last inequality is the statement of the lemma.
\end{proof}

 \begin{proof}[Proof of Theorem \ref{Thm-4}]
 Using the notations above, we do a translation which moves $(-1,-1,...,-1)$ to the origin. Then $P$ will be a polytope inside the first quadrant with barycenter inside $F$ by the second lemma. The rigidity follows from this together with the assumption that $\text{Vol}(P)=\text{Vol}(Q)$ by the first lemma.

 Now we consider the stability. Suppose the statement doesn't hold, then there is a sequence of manifolds $(M_i,\omega_i)(i=1,2,3...)$ with volume converging to $\text{Vol}(\mathbb{CP}^n)$ and none of them is holomorphic to $\mathbb{CP}^n$.

   Construct adapted pairs $(P_i,L_i)$ of $(M_i,\omega_i)(i=1,2,3...)$. Because there are only finitely many such $L$, one of them appears infinitely times. We denote it by $B$ and select these $P_i$ corresponding to $B$.
These $P_i$ as moment polytopes of different symplectic classes can be obtained from $B$ by parallel movement of $B$'s facets of towards the interior. So $P_i$ can be determined by $N$ real numbers $\lambda_i$ such that $n$ of them are always $-1$.  This gives us a correspondence:
 \begin{align}
 P_i \leftrightarrow \lambda^{(i)} \in\mathbb{R}^{N-n}.
 \end{align}
 Because $P_i$ are inside $B$, these vectors in $\mathbb{R}^{N-n}$ are bounded. We can choose a convergent subsequence, and the limit corresponds to a polytope $P_\infty$. $\text{Vol}(P_\infty)=\text{Vol}(Q)$ and the coordinates of the barycenter of $P_\infty$ are all nonpositive. According to the first lemma, $P_\infty$ should be isomorphic to $Q$ by a translation. We are going to show that $B=P_\infty$: Since $P_i\subseteq B$, we have $P_\infty\subseteq B$. If $P_\infty\varsubsetneqq B$, the integral points in the interior of the facet of $P_\infty$ opposite $(-1,-1,...,-1)$ will be contained in the interior of $B$. But there is only one integral point in the interior of $B$, so we must have  $B=P_\infty$.

 We assumed that $M_i$ are not holomorphic to $\mathbb{CP}^n$, but now $B$ just differs from $Q$ by a translation. It follows that these $M_i$ are all holomorphic to $\mathbb{CP}^n$. It's a contradiction, so our theorem is proved.
 \end{proof}

Department of Mathematical Science, Peking University, 

Beijing 100871, China.

Email address: fengwang232@gmail.com
\end{document}